\documentclass[11pt]{article}

\usepackage{amssymb,amsmath, amscd,amsfonts,amsthm,mathtools}
\usepackage{mathtools}

\usepackage[blocks]{authblk}

\usepackage{hyperref}
\usepackage{siunitx}
\usepackage{booktabs}

\usepackage[a4paper, left=3.2cm,top=2.5cm,right=3.5cm,bottom=3.7cm]{geometry}
\usepackage{cmbright}
\usepackage[T1]{fontenc}

\usepackage{enumitem}
\setitemize{label=\scriptsize{$\blacksquare$},topsep=0em,wide,itemsep=0em }
\setenumerate{label=(\alph*),topsep=0em}

\usepackage{booktabs}

\usepackage[utf8]{inputenc}
\usepackage[T1]{fontenc}

\usepackage{epstopdf}
\usepackage[english]{babel}
\usepackage{amsmath,amssymb}
\usepackage{dsfont}
\usepackage{mathtools}

\usepackage[color=green!40]{todonotes}

\usepackage{siunitx}

\usepackage{xcolor}
\colorlet{lred}{red!40}
\colorlet{lgreen}{green!40}
\colorlet{lblue}{blue!40}

\newcommand*{\N}{\mathds{N}}

\newcommand*{\R}{\mathds{R}}
\newcommand*{\C}{\mathds{C}}
\newcommand{\Ro}{\mathbf{R}}
\newcommand{\Mo}{\mathbf{M}}
\newcommand{\Bo}{\mathbf{B}}
\newcommand{\Do}{\mathbf{D}}
\newcommand{\Co}{\mathbf{C}}
\newcommand{\vv}{\mathbf{v}}
\newcommand{\ww}{\mathbf{w}}
\newcommand{\z}{\mathbf{z}}
\newcommand{\y}{\mathbf{y}}
\newcommand{\x}{\mathbf{x}}
\newcommand{\Base}{\boldsymbol \Psi}

\newcommand*{\transpose}[1]{{#1}^\intercal}

\newcommand*{\Int}[4]{\int_{#1}^{#2}\!{#3}\,\mathrm{d}{#4}}

\newcommand*{\e}{\mathrm{e}}

\theoremstyle{theorem}
\newtheorem{theorem}{Theorem}[section]

\theoremstyle{definition}
\newtheorem{definition}[theorem]{Definition}

\renewcommand{\phi}{\varphi}
\renewcommand{\epsilon}{\varepsilon}
\renewcommand{\theta}{\vartheta}

\DeclarePairedDelimiter{\abs}{\lvert}{\rvert}
\DeclarePairedDelimiter{\norm}{\lVert}{\rVert}
\DeclarePairedDelimiter{\innerprod}{\langle}{\rangle}

\newcommand{\F}{\mathbf{F}}
\DeclareMathOperator*{\mmin}{minimize}
\DeclareMathOperator*{\minimize}{minimize\,}

\DeclareGraphicsExtensions{.eps,.pdf,.png,.jpg}

\numberwithin{equation}{section}
\numberwithin{figure}{section}

\allowdisplaybreaks

\usepackage{soul}
\usepackage{xcolor}

\colorlet{lred}{red!40}
\colorlet{lgreen}{green!40}
\colorlet{lblue}{blue!40}
\definecolor{mixc}{cmyk}{0.5,0.5,0.5,0}
\colorlet{mixl}{mixc!30}

\begin{document}

\title{Compressive Time-of-Flight 3D Imaging Using Block-Structured Sensing Matrices}

\author[1,$\star$]{Stephan~Antholzer}
\author[1]{Christoph~Wolf}
\author[1]{Michael ~Sandbichler}
\author[2]{Markus~ Dielacher}
\author[1,$\star$]{Markus~Haltmeier}

\affil[1]{Department of Mathematics, University of Innsbruck\authorcr
Technikerstra{\ss}e 13, 6020 Innsbruck, Austria\vspace{1em}}

\affil[2]{Infineon Technologies Austria\authorcr
              Babenbergerstra{\ss}e 10, 8020 Graz, Austria\vspace{1em}}

\affil[$\star$]{Correspondence: {\{stephan.antholzer,markus.haltmeier\}@uibk.ac.at}}

\date{December 22, 2018}

\maketitle

\begin{abstract}
Spatially and temporally  highly resolved depth information  enables numerous applications including human-machine interaction in
gaming or safety functions  in the  automotive industry.  In this paper, we address this issue  using Time-of-flight (ToF) 3D cameras which
are compact devices providing highly resolved depth information.  Practical restrictions often require to  reduce the  amount of data to be read-out and transmitted. Using standard ToF cameras, this can only be achieved by lowering the spatial or temporal
resolution. To overcome such a limitation,  we propose a compressive ToF camera  design using block-structured sensing matrices that allows to reduce the amount of data while keeping high
spatial and temporal resolution.   We propose  the use of efficient reconstruction algorithms  based on
$\ell_1$-minimization and TV-regularization. The reconstruction methods are applied to data captured by a real ToF camera system and evaluated in terms of reconstruction quality and
computational effort.  For both, $\ell_1$-minimization and TV-regularization,  we use
a local as well as a global reconstruction strategy. For all considered instances,
global TV-regularization  turns out to clearly perform best in terms of evaluation metrics
including the  PSNR.

\end{abstract}

\section{Introduction}

 Time-of-Flight (ToF) camera systems rely on the time of flight (or travel time) of an emitted
 and reflected light  beam to create a depth image of a scenery. They offer
 many advantages over traditional systems (e.g. lidar) such as compact design,
registered depth and intensity images at a high frame rate, and low
power consumption~\cite{ToF-paper}.
This makes them ideal for mobile usage, for example, on a mobile phone. On such
devices, the computational resources for the required image reconstruction
algorithms are limited.
While there are several technologies allowing 3D imaging, in this paper we will focus on cameras that use a modulated
light source to calculate the phase shift (encoding the depth image)  between the emitted  and received signal~\cite{Hansard}.

High spatial and temporal resolution requires a  large amount of data to be read-out and transferred
from ToF cameras. In order to determine a depth image, typically four different  phase images per
frame have to be collected with the ToF camera. However, even from  four phase images,  the depth image
is unique only up to a certain maximal distance from the camera. To measure larger distances one needs additional  phase images that have to be read-out and transferred. Also in multi-camera systems, where the depth image is calculated outside the camera, the amount of data can be very high.
If the data rate or the read-out process is a limiting factor, either the spatial or the temporal resolution has to be reduced in a conventional ToF camera.

\begin{figure}[htb!]
\includegraphics[width=\columnwidth]{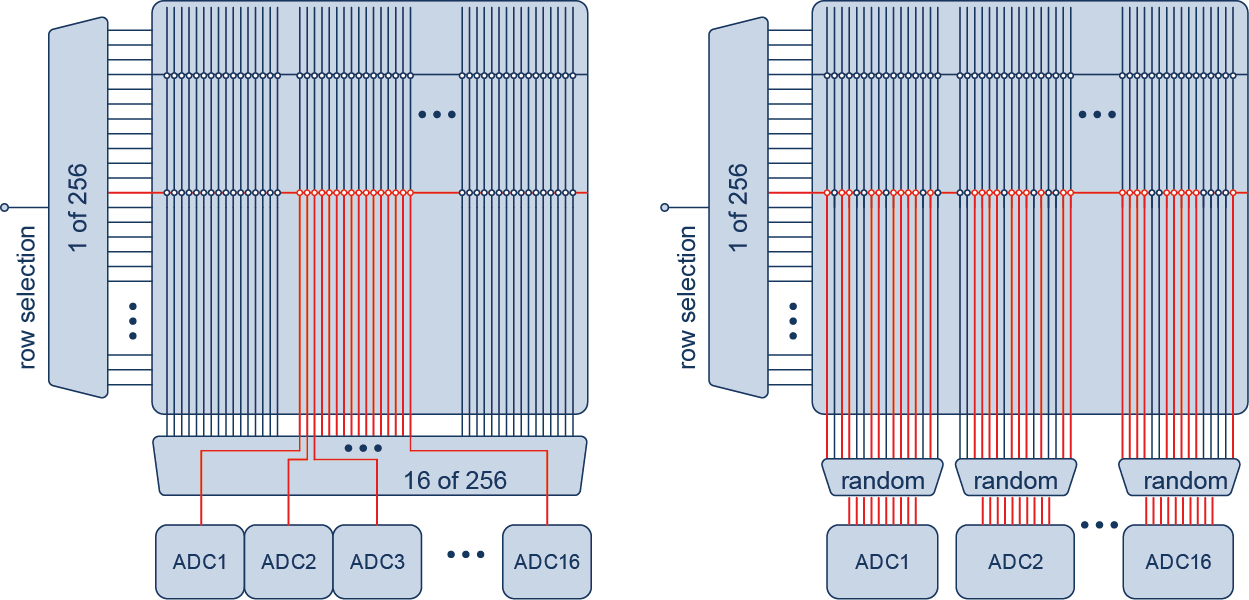}
\caption{\textbf{Standard read-out versus compressive read-out.} Left:
Standard sequential read-out using 16 ADCs used by our ToF camera.
Right: Proposed  compressive read-out. In each row, instead of 16 sequential read-outs per ADC,
$m/K$ combinations of pixel values are read-out. Note that compared to the standard  camera design, the 
main difference is the used multiplexer that combines the read-outs within a single block according the used 
measurement matrix.}  \label{fig:adc}

\end{figure}

\subsection*{Proposed compressive ToF imaging  approach}

To address the issues mentioned above, in this article we propose a compressive ToF camera that allows a reduced amount of data to be read-out and transferred  while preserving high spatial and temporal resolution.
Instead of individual pixels of the phase images, the compressive ToF camera reads out combinations of pixel values that are  transferred to an external processor.   As shown in Figure~{\ref{fig:adc}},  the only additional element that has to be added to the existing camera design  
is multiplexer per block, which combines the pixels read-out according to the used measurement design. Note that multiplexers are  standard element in CMOS (complementary metal-oxide-semiconductor) sensors \cite{zimmermann1997low}. Therefore, the proposed  camera
design  only requires  small modification of the existing camera design shown in Figure~{\ref{fig:adc}}, left. As in the existing used ToF camera we only use combinations of elements in the same row which yields  to block-structured sensing matrices. The actual manufacturing of such a compressive camera design  is beyond the scope of this paper. Note that by  using modern multi-layer circuit routing technology {\cite{katic2015compressive}}, one could also use combinations from different overlapping routing paths. We  restrict our ourselves to the block-structured sensing matrices because they are compatible with the existing ToF camera design.

In order to  reconstruct the  original phase and depth images we use techniques from sparse recovery using $\ell_1$-minimization and total variation (TV)-regularization. For both methods
we implemented a block based local approach  as well  as global approach. In all instances the
global TV-regularization turns out tot outperform the other tested reconstruction methods  in terms of
RMAE (relative mean absolute error), PSNR (peak signal to noise ratio) as well as visual  inspection. For example, using a compression ratio of 4.7,
global TV-regularization yields a PSNR of 31.5 for the recovered depth image of a typical
scenery, opposed to a PSNR of 26.4, 28.0 and 27.3  for block $\ell^1$-minimization,  global $\ell^1$-minimization  and block  TV-regularization. We address this to the following  issue.
Depth images and phase will typically be piecewise homogeneous.  Hence the
depth images  have  sparse  gradients  which exactly is what TV-regularization tries
to recover.

We point out that the proposed compressed sensing framework aims to accelerate the camera read-out and  not to reduce the  number of sensing  pixels itself, which would be another interesting line of research.
Different compressive ToF camera designs have been proposed
in~\cite{other-approach,conde2017compressive,Kadambi2015Lidar3D,Howland2011Lidar3D}.
The compressive designs in \cite{other-approach,Howland2011Lidar3D}
use a spatial light modulator and multiple pulses, whereas \cite{Kadambi2015Lidar3D} uses
coded aperture to gather multiple measurements.
In \cite{conde2017compressive}, compressed sensing ToF  imaging has been studied
in the spatial and temporal domain.
All these works use unstructured sensing matrices. In contrast to that, we use block-structured
sensing matrices which are easier to implement in existing ToF camera designs.
Additionally, the reduction of read-outs allows high spatial and temporal resolution.

Some  results  of this paper have  been presented at the International Conference on
Sampling Theory and Applications (SampTA) 2017  in Tallinn~\cite{antholzer2017compressive}.
The present  paper extends the  theoretical recovery results  stated in  \cite{antholzer2017compressive}
from the standard basis to a general  sparsifying  basis. Additionally, all numerical studies presented in this paper are new and extended significantly
compared to~\cite{antholzer2017compressive}. In particular, the TV-regularization studies for the proposed ToF  compressed sensing scheme
are completely new.

\subsection*{Outline}

In Section~\ref{sec:TOF} we give a short introduction the ToF imaging.
In Section~\ref{sec:CSTOF}  we present the type of measurements
that we propose for compressive ToF imaging.
We thereby start with details on the  classical (non-compressive) and the new (compressive) designs.
 Additionally, we prove that the used matrices
fulfill the RIP under suitable conditions. Moreover, in  Section~{\ref{sec:CSTOF}} we also
 introduce the   proposed block based   image reconstruction  approach using
 $\ell^1$-minimization and total variation. In the special cases of  single blocks, we obtain the global counterparts.  In Section \ref{sec:results} we give details on the numerical algorithm and
present extensive studies of our  two-step reconstruction approach of recovering the
 depth  image from the compressed measurements. We compare the block-based  as well as
 global version for $\ell^1$-minimization and TV-regularization.  The results are evaluated 
 in terms of  RMAE and  PSNR. Visually as well as in terms of these quality measures  global
 TV-regularization  outperforms the other methods in all instances.

 \begin{figure}[htb!]
     \centering
    \includegraphics[width=\columnwidth]{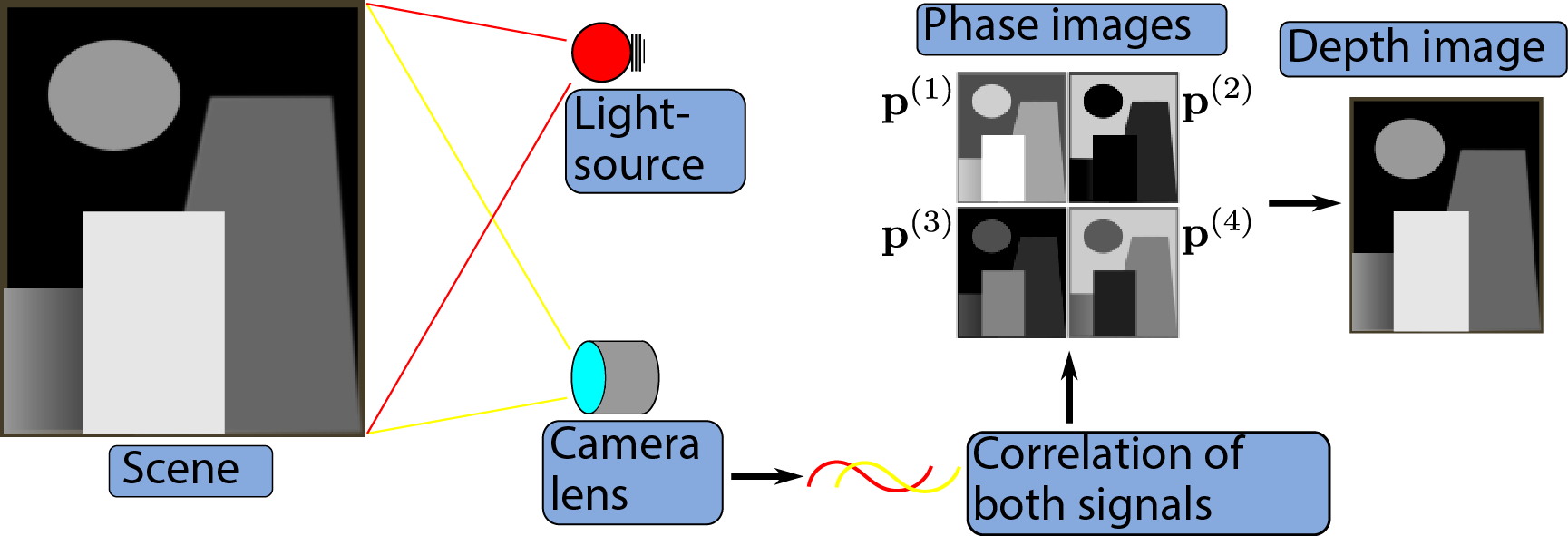}
    \caption{\textbf{Basic principles of a ToF camera}.
        Phase images are collected  by  sampling the cross-correlation of the emitted  with the
        reflected light pulse.
        From these phase images one computes a depth image mapping the distance from the scene to
    the camera.}
    \label{fig:ToFSetup}
\end{figure}

\section{Basics of 3D  imaging using ToF cameras}
\label{sec:TOF}

A ToF camera measures the distance of a scenery to the camera. By sending out a diffuse light pulse and
measuring the reflected signal, the camera is able to record depth information of the entire
scenery at once.
To acquire depth information, the sent out light is modulated and can be generated by an LED\@.
The scenery reflects the light which is recorded by the camera as depicted in Figure~\ref{fig:ToFSetup}. The emitted pulse can
be modeled as a time-dependent function $g(t) = C \cos(\omega t)$, where
$C$ is the amplitude, $\omega$ the modulation frequency
(or carrier frequency), and $t$ the time variable. The signal is reflected,
and the camera receives,  for any individual pixel $i \in \{1, \ldots , n\} $,
a phase and amplitude shifted signal
\begin{equation*}
f_i(t) = B_i + A_i \cos(\omega t - \phi_i) \,.
\end{equation*}
Here $\phi_i$ is the phase shift depending on the distance $d_i$
between the camera and the scene mapped at pixel $i$,  $A_i$ the amplitude depending on the
reflectivity,
and $B_i$ an offset.
The phase shift is related to the distance $d_i$
 via the relation  $d_i  = \phi_i \;  c / (2\omega)$.

At each pixel of the ToF camera, the cross-correlation between the reference  and the reflected
signal is measured, where the cross-correlation between two signals $f\colon\R\to\R$ and
$g\colon\R\to\R$ is given by
\begin{equation}\label{eq:cross_correlation}
    c_{f,g}(s) = \lim_{T\to\infty} \frac{1}{2T}\Int{-T}{T}{f(t)g(t+s)}{t}\,.
\end{equation}
In our case, $c_{f,g}(\cdot)$ can be calculated analytically~\cite{Albrecht,Lange,Hansard} which yields
\begin{equation}\label{eq:our_cross_correletation}
    c_{f_i,g}(s)
 = \frac{A_i C}{2}\cos(\omega s + \phi_i)  + K_i \,.
\end{equation}
Here $K_i$  incorporates constants accounting for noise and the background
image generated by ambient light.
By sampling the cross-correlation function at the sampling points
$s \in \{ 0, \pi/(2\omega),\pi/\omega, 3\pi/(2\omega)\}$ we get
four so-called phase images
\begin{align*}
    \mathbf{p}^{(1)} &= \phantom{-}\frac{\mathbf{A}C}{2}\cos(\boldsymbol{\phi}) + \mathbf{K}
    \\
    \mathbf{p}^{(2)} &=- \frac{\mathbf{A}C}{2}\sin(\boldsymbol{\phi}) + \mathbf{K}\\
    \mathbf{p}^{(3)} &= -\frac{\mathbf{A}C}{2}\cos(\boldsymbol{\phi}) + \mathbf{K}
    \\\mathbf{p}^{(4)}     &= \phantom{-} \frac{\mathbf{A}C}{2}\sin(\boldsymbol{\phi}) + \mathbf{K}
    \,.
\end{align*}
Here we have set $\boldsymbol{\phi}  \triangleq (\phi_j)_{j=1}^n \in \R^n$, $\mathbf{K}  \triangleq
(K_j)_{j=1}^n$ and $ \mathbf{A} \triangleq (A_j)_{j=1}^n  \in \R^n$,  and all   operations are taken  point-wise.
Under the common  assumption that $K_i$ is independent  of the
pixel location we can estimate the phase shifts $\boldsymbol \phi$ by
\begin{equation}\label{eq:phase}
\hat{\boldsymbol{\phi}} = \operatorname{\arg} \left(\mathbf{p}^{(1)}-\mathbf{p}^{(3)} + i \bigl(
\mathbf{p}^{(4)}-\mathbf{p}^{(2)} \bigr) \right)   \,.
\end{equation}
Here  $\alpha = \operatorname{\arg} (z) \in [0, 2\pi)$ denotes the argument of the complex number
$z$ defined by $z  \triangleq  r e^{i \alpha}$.    In particular, the depth image is given by
$\hat{\mathbf{d}} = \hat{\boldsymbol{\phi}} \; c/(2\omega)$.

Since the phase shifts   are contained  in $[0,2\pi)$, the maximal distance that can be found by
\eqref{eq:phase}   is $d_{\max} = (c /\omega) \pi$. Larger distances are falsely identified, taking
values in the interval
$[0,d_{\max})$.        To overcome  this ambiguity, several methods  have been proposed in the
literature (see, for example, ~\cite{Hansard,Droeschel}).  One such approach  consists in
capturing two sets of phase images with different  modulation frequencies $\omega_1 \neq \omega_2$,
and then comparing the two depth images. In this paper, we will not address the ambiguity problem further.
The   compressive ToF camera that we propose below  can be extended  to multiple modulation frequencies in a straightforward manner.
Sparse recovery can be applied to any of the phase images, even if they are affected by phase wrapping.
Further research, however, is needed to thoroughly investigate such  extensions,
in particular, to investigate sparsity issues, accuracy  and noise stability.

\section{Compressive ToF sensing and image reconstruction}
\label{sec:CSTOF}

In this section we present the  proposed   compressive ToF 3D sensing design
compatible  with existing ToF cameras. Additionally, we describe an efficient
block-wise reconstruction procedure based on sparse recovery.

\subsection{Compressive ToF sensing}
\label{sec:read-out}

As mentioned in the introduction, in a conventional ToF camera, all  pixel values of all phase images  have to be read-out
and large amounts of data have to be transferred. To reduce the amount of data, in this paper, we propose a
compressive ToF camera, which reads out and transmits linear combinations instead of individual pixel values of the phase image.
Our proposed compressive ToF camera design is based on the existing non-compressive ToF camera
design, which should allow to engineer and build the new camera with low effort.
The only difference between the two designs is in the way the pixels of the sensors are read-out.
For the compressive ToF camera, we propose to read-out linear combinations of neighboring
pixels.

The data collected by the compressive ToF camera can be written in the form
\begin{equation*}
\forall i \in \{1,2,3,4\} \colon \quad
\mathbf{y}^{(i)}
=
\Mo \mathbf{p}^{(i)} \,.
\end{equation*}
Here $\Mo \in\R^{m\times n}$ is the measurement  matrix,
$\mathbf{p}^{(i)}\in\R^n$ are the phase images and $\mathbf{y}^{(i)}\in\R^m$ the read-out data with $m \ll n $.  To reconstruct the depth image from the compressed read-outs we propose the  following two-step procedure: First, we  estimate  the differences
$\mathbf{p}^{(1)} - \mathbf{p}^{(3)}$ and $\mathbf{p}^{(4)}-\mathbf{p}^{(2)}$ from
\begin{align} \label{eq:diff1}
 \mathbf{y}^{(1)} - \mathbf{y}^{(3)}
& = \Mo  \left( \mathbf{p}^{(1)} - \mathbf{p}^{(3)}  \right) \,,
\\ \label{eq:diff2}
\mathbf{y}^{(4)}-\mathbf{y}^{(2)}
& = \Mo  \left( \mathbf{p}^{(4)}-\mathbf{p}^{(2)}  \right) \,,
 \end{align}
using sparse recovery. In a second step we recover the depth image by applying~\eqref{eq:phase} to the estimated differences.

Any of the equations \eqref{eq:diff1}, \eqref{eq:diff2}
is an   underdetermined system  of the form $\mathbf{y} =\Mo \mathbf{x}$, for which in general no unique solution exists.  To obtain solution uniqueness, the vector $\mathbf{x} \in \R^n$ needs to satisfy certain additional requirements.
In recent years, sparsity turned out to be a powerful property for this purpose. Recall that  $\mathbf{x}$ is called $s$-sparse, if it has at most
$s$ nonzero entries. Assuming sparsity,  the vector $\mathbf{x}$ can, for example, be
recovered by solving the $\ell_1$-minimization problem
\begin{equation}\label{eq:ell1}
\mmin_{\mathbf{z}\in\R^n}\norm{\mathbf{z}}_1 \quad \text{subject to } \Mo \mathbf{z} = \mathbf{y} \,.
\end{equation}
In order for \eqref{eq:ell1} to uniquely recover $\mathbf{x}$,
the matrix $\Mo$ needs to fulfill certain properties.
One sufficient condition is the restricted isometry property (RIP).
The matrix $\Mo$ is said to satisfy the  $s$-RIP with
constant $\delta>0$, if
\begin{equation*}
(1-\delta)\norm{\mathbf{z}}_2^2 \leq \norm{\Mo\mathbf{z}}_2^2 \leq (1+\delta)\norm{\mathbf{z}}_2^2
\end{equation*}
holds for all $s$-sparse $\mathbf{z} \in \R^n$. If the $s$-RIP
constant is sufficiently small, then \eqref{eq:ell1} uniquely recovers
any sufficiently sparse vector  (see, for example, ~\cite{CanRomTao06b,CS,cai2014sparse}).

Although some results for deterministic RIP matrices exist~\cite{simple_det_mm,explicit_RIP},  matrices satisfying the RIP are commonly constructed in some random manner. Realizations of
Gaussian or Rademacher random matrices  are known to satisfy the RIP with high probability~\cite{CS}.
Rademacher random variables take
the values -1 and 1 with equal probability. It has been
shown \cite{candes2006near} that if $m\geq C \delta^{-2} s\log(n/s)$, then $\delta_s\leq \delta$
with high probability for both types of matrices.
Note that up to the logarithmic factor, this bound  scales linearly in the sparsity level $s$.
In this sense, the bound is optimal in the sparsity level, because $m \geq 2s$ is the
minimal requirement  for any measurement matrix to be able to recover all $s$-sparse vectors~{\cite{CS}}.
More generally, all
matrices with independent sub-Gaussian entries satisfy the RIP with high probability. A sub-Gaussian random variable
$X$ is defined by the property
\begin{equation}\label{eq:subgauss}
    \mathds{P}(\abs{X}\geq t) \leq \beta \e^{-\kappa t^2} \quad \text{for all } t>0
\end{equation}
with constants $\beta,\kappa>0$. It is easy to show that Rademacher and Gaussian random variables are
sub-Gaussian. Sub-Gaussian matrices $\Mo$ are also universal~\cite{CS} which means that for any unitary
matrix $\Base\in\C^{n\times n}$ the matrix $\Mo\Base$ also satisfies the RIP\@. Thus one can also
recover signals that are not sparse in the standard basis but for which $\Base^\ast \x$
is sparse, where $\Base^\ast$ is the conjugate transpose. This property is very useful in
applications
since many natural signals have sparse representations in certain bases different from the standard
basis. In general, if the
restricted isometry constant of
$\Mo\Base$ is small, then we can recover the signal by solving \eqref{eq:ell1} with
$\Mo\Base$ instead of $\Mo$ and, in the noisy case, by
\begin{equation}\label{eq:tikhonov}
    \minimize_{\z\in\C^n} \norm{\Base^\ast \z}_1 \quad \text{such that } \norm{\Mo\z-\y}_2\leq \eta
    \,.
\end{equation}
Here $\eta$ equals the noise level of the measurements.
Similar results hold if $\Base$ is a
frame (see \cite{CanEldNeeRan11,Hal13b,redundant_dictionaries,RauSchVan08}).

In practical applications, such unstructured matrices 
cannot always be used. Either there are restrictions on
the matrix preventing us from using a random matrix with i.i.d. entries or
the storage space is
limited, such that storing a full matrix would be too expensive.
There are different methods for
constructing structured compressed sensing matrices that satisfy the RIP. For example, such
matrices can be constructed by random subsampling of an orthonormal matrix~\cite[Chapter
12]{CS} or
deterministic convolution followed by random subsampling~\cite{li2013convolutional} or
using a random convolution followed by deterministic subsampling~\cite{rauhut2012restricted}. In
the next section we will examine the latter type since its application to ToF imaging and the
existing camera designs   they have beneficial properties. More specifically, such measurement matrices  can be constructed to only use  entries form  a small set, such as $\{-1,0 ,1\}$,  
can efficiently be implemented  by using  Fourier transform techniques, and 
require litte  information to be stored.

\subsection{Compressive 3D Sensing Using Block Partial Circulant  Matrices}
\label{ssec:circulant}

The hardware requirements in our case prevent us from using arbitrary matrices since for the
analog-to-digital converters
(ADC) the weights 0 and $\pm a$ for some fixed constant  $a \in\R$
should be used. Further,  any individual ADC can only be wired with a limited number of pixels
(compare Figure~\ref{fig:adc}) which imposes a particular block-structure  of the measurement matrix.
Thus, the  measurement matrices that we use in our approach take the block-diagonal form
\begin{equation} \label{eq:block}
\Mo =\Mo_1 \otimes \cdots \otimes \Mo_K
\triangleq
\begin{pmatrix}
\Mo_1 & \mathbf{0} & \cdots & \mathbf{0} \\
\mathbf{0} & \Mo_2 & \cdots & \mathbf{0}\\
\vdots & & \ddots & \vdots \\
\mathbf{0} & \cdots & & \Mo_K
\end{pmatrix}\,.
\end{equation}
Here, each sub-matrix $\Mo_k \in \R^{m_k \times n_k}$ operates on a certain subset $\Omega_k \subsetneq \{1,\ldots,  n\}$ with $n_k=\abs{\Omega_k}$ elements coming from a single row in the image.  For simplicity we consider the case that
$n_k  = n/K$ and  $m_k  = m/K$ for each $k$.
The particular measurements in each row block are constructed in a certain
random manner satisfying the requirements above.

A particularly useful class of such row-wise measurements in the compressive ToF camera
can be modeled by  partial circulant  matrices.  A
circulant matrix $\Co_\vv \in \R^{n\times n}$ associated with $\vv =(v_1, \ldots,  v_n) \in  \R^n$ is defined  by
\begin{equation*}
\forall i,j \in \{1, \dots, n\} \colon \quad
	(\Co_\vv) _{i,j} = v_{i\ominus j} \,,
\end{equation*}
where $j\ominus i \triangleq  (j-i) \mod n$ is the cyclic subtraction.
In particular, for all $\vv, \ww\in\R^n$, we have
$\Co_\vv \ww = \vv \ast \ww$, where
$(\vv \ast \ww)_j \triangleq \sum_{i=1}^n  v_{j\ominus i} w_i$ is
 the circular convolution.  For any subset $\Omega \subseteq \{1,\ldots , n\}$,
 the projection matrix $\Ro_\Omega  \in \R^{|\Omega|\times n} $  is defined
 by $\Ro_\Omega \vv \triangleq (\vv_i)_{i\in \Omega}$.

\begin{definition}
The  partial circulant  matrix  associated to  $\vv \in \R^n$  and
 $\Omega \subseteq \{1,\ldots, n\}$ is defined by
 $\frac{1}{  \sqrt{\abs{\Omega}} } \Ro_\Omega \Co_{\vv}$.
\end{definition}

Further, recall that a random vector $\vv$  with values in $\{\pm 1\}^n$ is called a
Rademacher vector if it has independent entries taking the values  $\pm 1$
with equal probability.
Partial circulant matrices satisfy the RIP. Such results have been obtained first
in~\cite{rauhut2012restricted} and have later been refined in~\cite[Theorem~1.1]{suprema} using the theory
of suprema of chaos processes. These results have been formulated for sparsity in the standard basis.
For our purpose, we formulate such a result  for the general orthonormal bases.

\begin{theorem} \label{thm:cyclic}
    Consider the partial circulant  matrix  $\Mo = \frac{1}{  \sqrt{m} } \Ro_\Omega \Co_{\vv}$
    associated to a  vector $\vv$ with iid sub-Gaussian entries and  a subset
    $\Omega \subseteq \{1,\ldots , n\}$ containing $m$ elements. If, for some $s\leq n$
    and $\delta \in (0,1)$, we have
    \[
        m \geq C \delta^{-2}\mu^2s(\log s)^2(\log n)^2 \,,
    \]
    then, with probability at least $1-n^{-\log n \log^2s}$, the $s$-RIP
    constant of $\Mo\Base$ is at most $\delta$, where the constant $\mu$ is given by
    $\mu = \max_{i,j}\abs{\innerprod{\F_{j-},\Base_{i-}}}$. Here $\F$ is the discrete Fourier
    matrix and $\Base \in \C^{n \times n}$ is any unitary matrix.
    \end{theorem}

\begin{proof}
For the case that $\Base$ is the identity matrix the result is derived in the original paper~\cite{suprema}.
The generalization to arbitrary $\Base$ can be shown analogously to the original results.
Such a proof is worked out in~\cite{antholzer2017nonlinear,james}.
\end{proof}

Theorem~\ref{thm:cyclic} shows that  random partial circulant  matrices yield stable recovery of sparse vectors using~\eqref{eq:ell1}.
Recall that the proposed compressive  ToF camera read-out uses block diagonal measurement matrices
of  the form \eqref{eq:block}. Taking each block as a random partial circulant matrix and applying Theorem~\ref{thm:cyclic} yields the following result.

\begin{theorem}\label{thm:ourtheorem}
Let $\Mo\in\R^{m\times n}$ be of the form \eqref{eq:block}, where each block on the diagonal is a partial circulant  matrix  $\Mo_k = \frac{1}{\sqrt{m_k}} \Ro_{\Omega_k} \Co_{\vv_k}$ associated with independent   Rademacher vectors $\vv_k$  and   subsets $\Omega_k \subseteq \{1,\ldots , n/K\}$ having $m_k=m/K$ elements that are selected independently and uniformly at random.
If, for some $s\in \N$  and $\delta\in (0,1)$, we have
\[
m \geq K C \delta^{-2}  \mu^2 s (\log  s)^2 (\log(n/K))^2 \,,
\]
then, with probability at least $(1-(n/K)^{-\log (n/K)\log^2 s})^K$,
$\Mo_k \Base$ has the $s$-RIP constant of at most
$\delta$ for all $k=1,\ldots, K$. Here the constant $\mu$ is given by
    $\mu = \max_{i,j}\abs{\innerprod{\F_{j-},\Base_{i-}}}$ and
     $\Base \in \C^{(n/K) \times (n/K)}$ is any unitary matrix.
\end{theorem}

\begin{proof}
As $m/K \geq C\delta^{-2} s(\log  s)^2(\log(n/K))^2$,
we can apply Theorem~\ref{thm:cyclic} with
$n$ and $m$ replaced by $n/K$ and $m/K$  to each block. Thus the restricted isometry constant of each block is at most $\delta$ with probability at least $1-(n/K)^{-\log (n/K) \log^2 s}$. As the generating Rademacher vectors for  each block are independent, the  $s$-RIP constants of all blocks  are uniformly bounded
by $\delta$ with probability at least $(1-(n/K)^{-\log (n/K)\log^2 s})^K$.
\end{proof}

Theorem~\ref{thm:ourtheorem} yields stable recovery via \eqref{eq:tikhonov} if the vector
$\mathbf{x}$ is $\Base$-block sparse, meaning that $\mathbf{x} =
\left[\mathbf{x}_1,\ldots,\mathbf{x}_K\right]$ with  $\Base^\ast \mathbf{x}_k \in \R^{n/K}$ being $s$ sparse for all $k=1,\ldots,K$.

Note that a different compressive CMOS sensor design using partial circulant matrices has been proposed in~{\cite{CMOS2009}}. Our design uses multiple ADCs (actually,  the existing camera design suggests 16 ADCs) each of them  operates on a small number of pixels. Thus our design  uses   parallel read-out which should yield  to a faster imaging image acquisition speed.
However, at the same time, the structured read-out matrices  require an increased number of measurement.  Investigations  on the good choices of block sizes and comparison with other  CMOS sensor design is interesting line of  future research.

\subsection{Image reconstruction by $\ell_1$-minimization}
\label{ssec:alg}

As presented in Section~\ref{sec:read-out}, the  depth image  is recovered from compressed read-outs  by first estimating   the differences
$\mathbf{p}^{(1)} - \mathbf{p}^{(3)}$ and $\mathbf{p}^{(4)}-\mathbf{p}^{(2)}$ from \eqref{eq:diff1} and \eqref{eq:diff2}, which  are  underdetermined systems of equations of the form $\mathbf{y} =\Mo \mathbf{x}$,
and then applying
\eqref{eq:phase}  to the estimated differences.    In this subsection we
present  how to efficiently solve these underdetermined systems using block-wise  $\ell_1$-minimization.

Suppose that the measurement matrix $\Mo\in\R^{m\times n}$ has a block diagonal form \eqref{eq:block}, with diagonal blocks $\Mo_k \in \R^{(m/K) \times (n/K)}$ operating  on a subset of pixels from individual lines.
This type of measurement matrices reflects the current  ToF camera architecture illustrated
in Figure~\ref{fig:adc}.  Assuming the sparsifying basis $\Base$  to be block diagonal with diagonal
blocks  $\Base_k \in \R^{(n/K) \times (n/K)}$, the full $\ell_1$-minimization problem
\eqref{eq:tikhonov} can be decomposed
into $K$ smaller $\ell_1$-minimization  problems of the form
\begin{equation}\label{eq:ell1block}
\mmin_{\z\in\R^{n/K}}\norm{ \Base_k^* \z}_1 \quad \text{subject to } \norm{\Mo_k \z_k - \y_k}_2\leq
\eta/K \,.
\end{equation}
Here
$\y_k \in \R^{m/K}$  are the data from
a single block. If  all $\Mo_k$ satisfy the $\Base$-RIP (i.e. $\Mo_k\Base_k$ satisfies the RIP), then   \eqref{eq:ell1block} stably and
robustly recovers any $\Base-$block-sparse vector $\mathbf{x} = \left[\mathbf{x}_1,\ldots,\mathbf{x}_K\right]$ with
$\mathbf{y} = \Mo \mathbf{x}$.
Theorem~\ref{thm:ourtheorem} shows that this, for example,
is the case if $\Mo_k$ are realized as random partial  circulant matrices.

By solving \eqref{eq:ell1block} we exploit sparsity within a single row-block. While one can expect some row-sparsity,
\eqref{eq:ell1block} does not fully exploit the level of sparsity present in two-dimensional images.
As shown in~\cite{Wolf}, using  row-sparsity yields artifacts in the  reconstructed image. In this work, we therefore
follow a different approach that is  described next. For that purpose, we consider an additional partition of all pixels
\begin{equation} \label{eq:part}
\{ 1, \ldots ,   n\}
= \bigcup_{\ell  = 1, \dots, n/b^2}  B_\ell \,,
\end{equation}
where $B_\ell$ corresponds to all indices in squared blocks of
size $b\times b$ with $b \triangleq n/K$. Then the measurement matrix can be written in the form
$\Mo= \Bo_1 \otimes \cdots \otimes \Bo_{n/b^2}$  with diagonal blocks
\begin{equation}\label{eq:Bell}
\Bo_{\ell}  \triangleq
\begin{pmatrix}
\Mo_\ell & \mathbf 0 & \ldots & \mathbf 0 \\
\mathbf 0 & \Mo_{\ell+b} & \ldots & \mathbf 0 \\
\vdots & & \ddots & \vdots \\
\mathbf 0 & \ldots & & \Mo_{\ell+(b-1)b}
\end{pmatrix} \in \R^{b^2 \times b^2}
\end{equation}
for $\ell = 1, \ldots, n/b^2$.
We further assume that the sparsifying basis $\Base  = \Base_1 \otimes \cdots \otimes \Base_{n/b^2}$   is
block diagonal with  $\Base_k \in \R^{b^2 \times b^2}$.
In such a situation,   \eqref{eq:ell1} can be decomposed into
$n/b^2$ smaller $\ell_1$-minimization  problems,
 \begin{equation}\label{eq:ell1block2}
\mmin_{\z_\ell \in\R^{b^2}}
\norm{\Base_\ell^* \z_\ell}_1
\quad \text{subject to } \norm{\Bo_\ell \z_\ell - \y_\ell}_2
\leq \varepsilon \,.
\end{equation}
The advantage of \eqref{eq:ell1block2} over \eqref{eq:ell1block} is that  $\Base_\ell$ can now be chosen as a two-dimensional wavelet or cosine transform instead of their one-dimensional analogous.  Two-dimensional wavelet or cosine transform   are well known to  provide
sparse representations of images.   In particular, the sparsity level relative to the number of
number of measurements is larger than in the  one-dimensional case. A similar argumentation 
applies when a TV is applied as sparsifying transform.

On the other hand, \eqref{eq:tikhonov}  is still decomposed into smaller
subproblems  which enables efficient numerical implementations.  The optimization problems
\eqref{eq:ell1block2} can be solved in parallel which further decreases
computation times. Using a global sparsifying transformation might be better in terms of sparsity,
but the resulting problem is less efficient to solve.

For the actual numerical implementation we use $\ell_1$-Tikhonov-regularization
\begin{equation}\label{eq:block_tikhonov}
\minimize_{\z_\ell \in\R^{b^2}}
\lambda \norm{\z_\ell}_1
+\norm{\tilde\Bo_\ell \z_\ell - \y_\ell}_2^2
\end{equation}
where $\tilde \Bo_\ell=\Bo_\ell\Base_l$ can be calculated by
$\transpose{(\transpose\Base_\ell\transpose\Bo_\ell)}$.
The two problems \eqref{eq:ell1block2}, \eqref{eq:block_tikhonov} are equivalent~\cite{grasmair2011necessary}
in the sense that every solution of
\eqref{eq:ell1block2} is also a solution of \eqref{eq:block_tikhonov} for
$\lambda$ depending on $\epsilon$ and vice versa.
For minimizing the unconstrained
$\ell_1$-problem \eqref{eq:block_tikhonov} we use the  fast iterative soft thresholding algorithm (FISTA)
introduced in~\cite{fista}, a very efficient splitting algorithm for $\ell_1$-type
minimization problems. For the numerical results, we also consider  global $\ell_1$-minimization,
where  we apply \eqref{eq:block_tikhonov} to the  complete phase difference images.

\subsection{TV-regularization}
\label{ssec:tv}

In many imaging applications, total variation (TV)-regularization~\cite{SchGraGroHalLen09} has been shown to
outperform wavelet based  $\ell^1$-minimization for compressed
sensing~\cite{krahmer2014stable,poon2015role,michaelTV17}.
Therefore, as an alternative approach for reconstructing the
blocks $\x_\ell$  of the phase difference images,   we implemented TV-regularization
\begin{equation}\label{eq:blockTV}
\minimize_{\z_\ell \in\R^{b^2}}
\mu \norm{\Do  \z_\ell}_1
+\norm{\Bo_\ell \z_\ell - \y_\ell}_2^2 \,.
\end{equation}
Here   $\Do \colon  \R^{b^2} \to (\R^{b^2})^2$ denotes the
discrete gradient operator and $\mu$ is the -regularization parameter.
For the numerical results we also considered  global TV-regularization,
where  we use \eqref{eq:blockTV} without partitioning  into  individual  blocks.

In order to numerically solve  \eqref{eq:blockTV}, we use the primal dual
algorithm  of Chambolle and Pock \cite{chambolle2011first}.
While the above theoretical results cannot be applied for  \eqref{eq:blockTV},
similar to other applications, we found TV to empirically yield better results
than $\ell^1$-minimization for ToF imaging.

\begin{figure}[htb!]
    \centering
    \includegraphics[width=0.75\columnwidth]{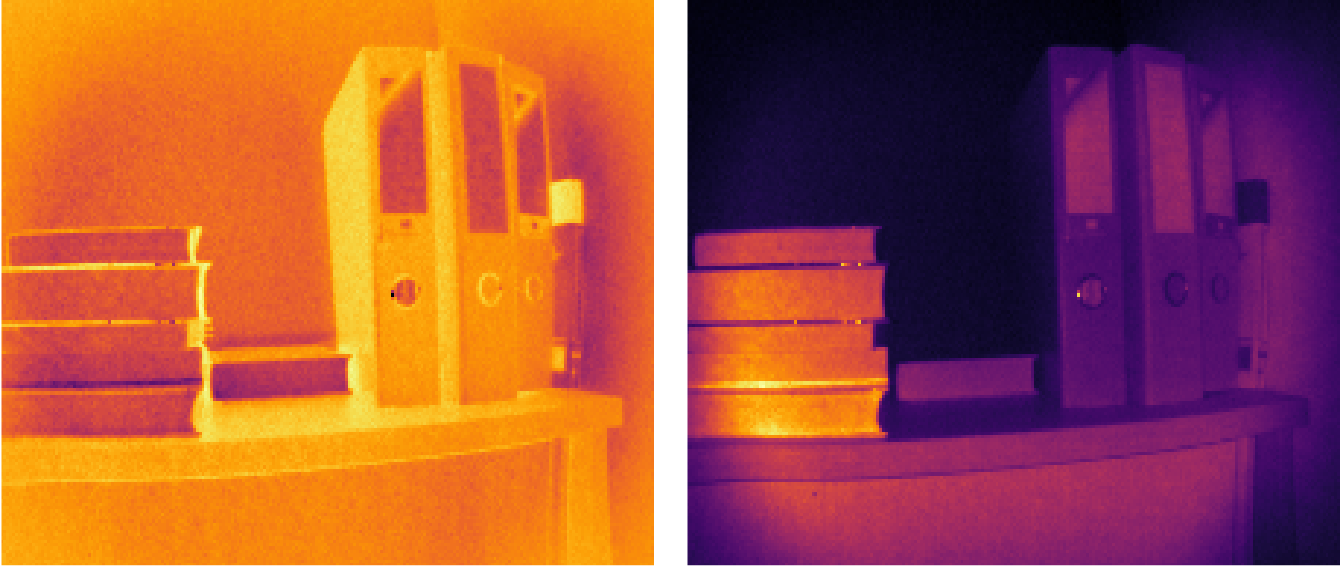}
    \caption{\textbf{Original phase difference images (of books scenery).}
    The four phase images are captured by the real  existing ToF camera and
    the two difference image are subsequently computed according to \eqref{eq:diff1}, \eqref{eq:diff2}.}
    \label{fig:phase1}
\end{figure}

\begin{figure}[htb!]
    \centering
    \includegraphics[width=0.6\columnwidth]{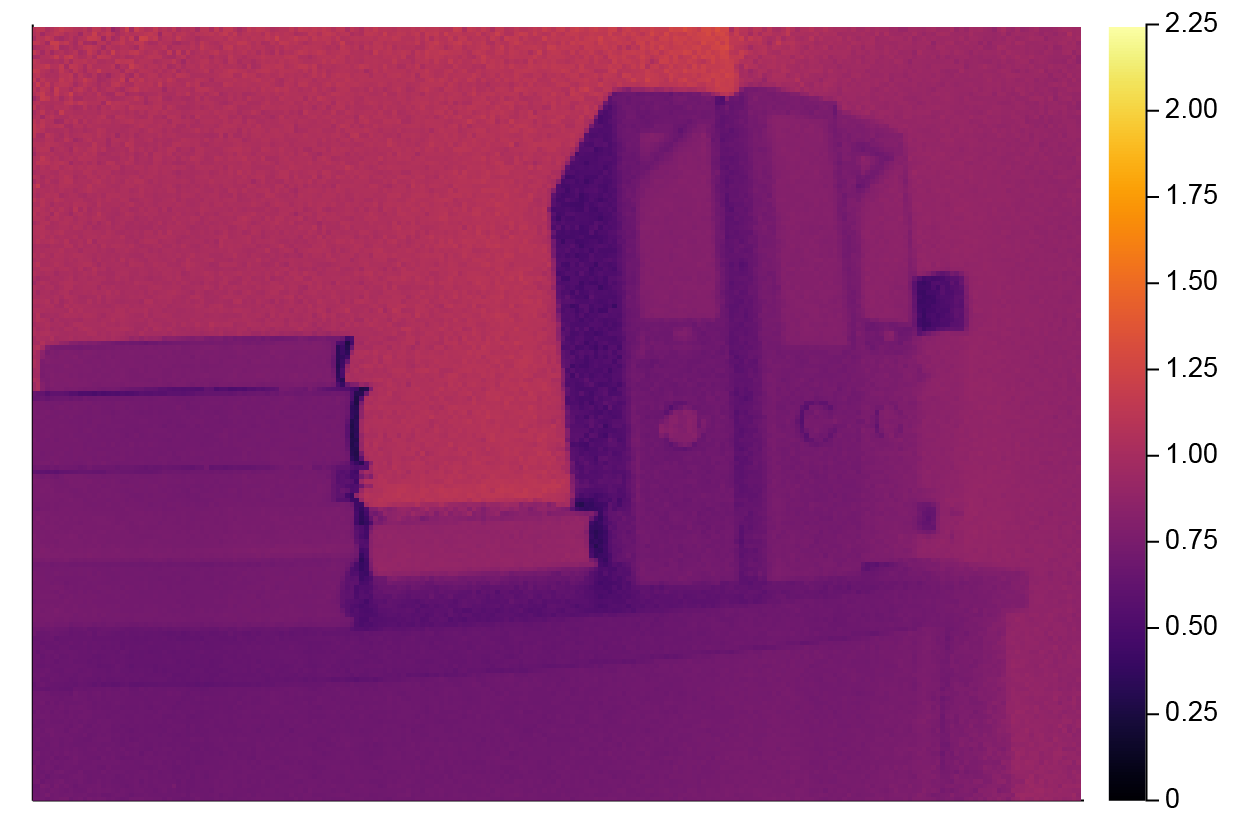}
    \caption{\textbf{Depth images.}
    The depth image is calculated from the two phase difference images shown in
     Figure~\ref{fig:phase1}.}
    \label{fig:depth1}
\end{figure}

\section{Experimental Results}
\label{sec:results}

In this section, we present some experimental results using raw data captured by an existing standard ToF camera.
An example of  such data (phase difference images of  books scenery) is shown in
Figure~\ref{fig:phase1}; the  depth image computed from the difference images via \eqref{eq:phase} is shown in Figure~\ref{fig:depth1}.
From the raw data of the standard ToF camera,  we generate the
compressive sensing measurements  synthetically.
For image reconstruction from compressed sensing data we use the
block-wise and global $\ell_1$-minimization \eqref{eq:block_tikhonov}
as well as block-wise and global TV-regularization\eqref{eq:blockTV}.

\subsection{Compressed ToF sensing}

For compressive ToF sensing, we initialized the measurement matrices $\Mo$ (the block
circulant matrices; see
Section~\ref{ssec:circulant}) randomly with the entries of a random vector generating the
partial circulant blocks taking values in $\{-1,1,0\}$ with equal probability. The blocks have size
$m\times 14$ which implies that the compression ratio is $14/m$.
In the experiments we observed that
usually not all blocks of our measurement matrix yield adequate reconstruction
properties.
This indicates that the size of the single blocks is not large enough to
guarantee recovery in each block with high probability. Using bigger blocks would
overcome this issue (according to Theorem~\ref{thm:ourtheorem}), but this is not
possible with our camera design. We therefore propose the following  alternative strategy.
We start with a set of several candidates for the blocks of the measurement matrix
from which we choose the ones with the lowest reconstruction error on a set of test images.

For the following results we have chosen the parameters in the FISTA for  $\ell_1$-minimization and
the TV algorithm by hand and did not perform extensive parameter optimization.
On most images the parameter choice had a moderate influence on the reconstruction error.
Thus we used $\lambda=0.05$ and $\mu =0.1$ for all
presented results. For the basis $\Base$ we use the 2D-Haar wavelet
transform and, as described in Section~\ref{ssec:alg}, we executed the reconstruction block
wise with a  block size of $28\times 28$. Additionally, for $\ell_1$-minimization as well as TV-regularization we
performed global reconstruction corresponding to a single block.
We use  300 iterations for the block wise $\ell_1$-minimization,  1000 for the
global  $\ell_1$-minimization, 100 for block wise TV, and 300 for global TV.

To measure the error between the uncompressed depth image
$\mathbf{d}\in\R^{n_1\times n_2} = \R^{168\times 224}$ and the
reconstructed depth image $\mathbf{d}_{\rm rec}\in\R^{168\times 224}$ we use
the relative mean absolute error (RMAE) and the peak signal to noise ratio (PSNR)
defined by
\begin{equation*}
    \begin{aligned}
        \text{MAE}(\mathbf{d},\mathbf{d}_{\rm rec}) &\triangleq \frac{1}{n_1 n_2}
        \sum_{i=1}^{n_1}\sum_{j=1}^{n_2} \abs{d_{i,j} -(d_{\rm rec})_{i,j}}\\
        \text{RMAE}(\mathbf{d},\mathbf{d}_{\rm rec}) &\triangleq
        \frac{\text{MAE}(\mathbf{d},\mathbf{d}_{\rm rec})}{\max_{i,j}\abs{d_{i,j}}} \, \SI{100} {\percent}\\
        \text{PSNR}(\mathbf{d},\mathbf{d}_{\rm rec}) &\triangleq 10 \log \left(
        n_1n_2 \frac{\max_{i,j}
        (d_{i,j}^2)}{\sum_{i,j} (d_{i,j}-(d_{\rm rec})_{i,j})^2} \right)\SI{}{\decibel}\,.
    \end{aligned}
\end{equation*}

\begin{figure}[htb!]
    \centering
    \includegraphics[width=0.75\columnwidth]{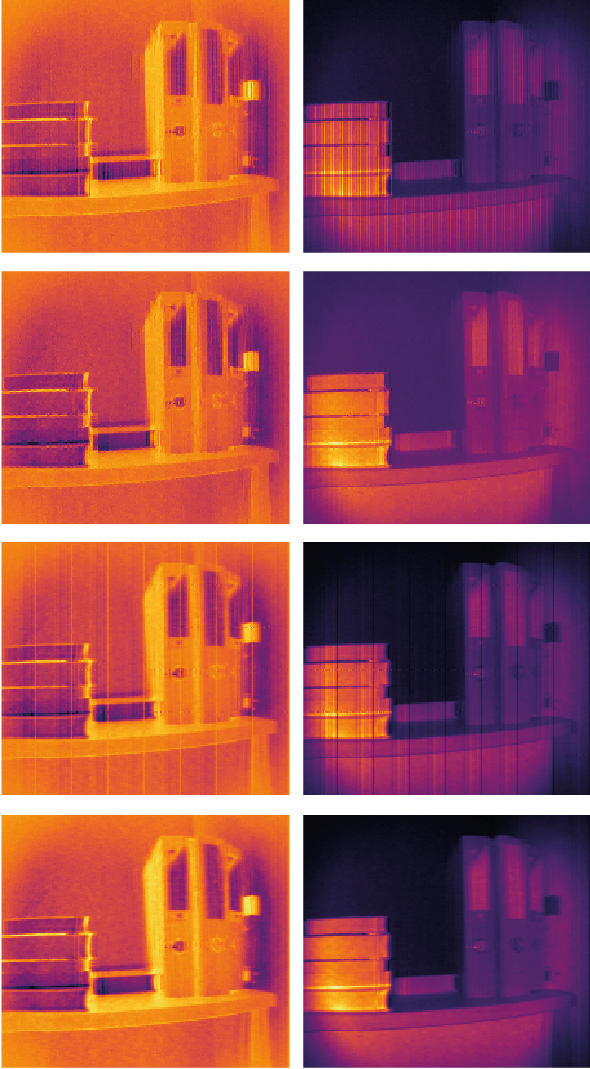}
    \caption{\textbf{Reconstructed phase differences with a
        compression ratio of $2$.}
        First row: Block-wise  $\ell_1$-minimization.
        Second row: Global  $\ell_1$-minimization.
        Third row:  Block-wise TV-regularization.
        Fourth row: Global TV-regularization.
        The original phase differences are shown in Figure~\ref{fig:phase1}.}
    \label{fig:phase-recon}
\end{figure}

\begin{figure}[htb!]
    \centering
    \includegraphics[width=\columnwidth]{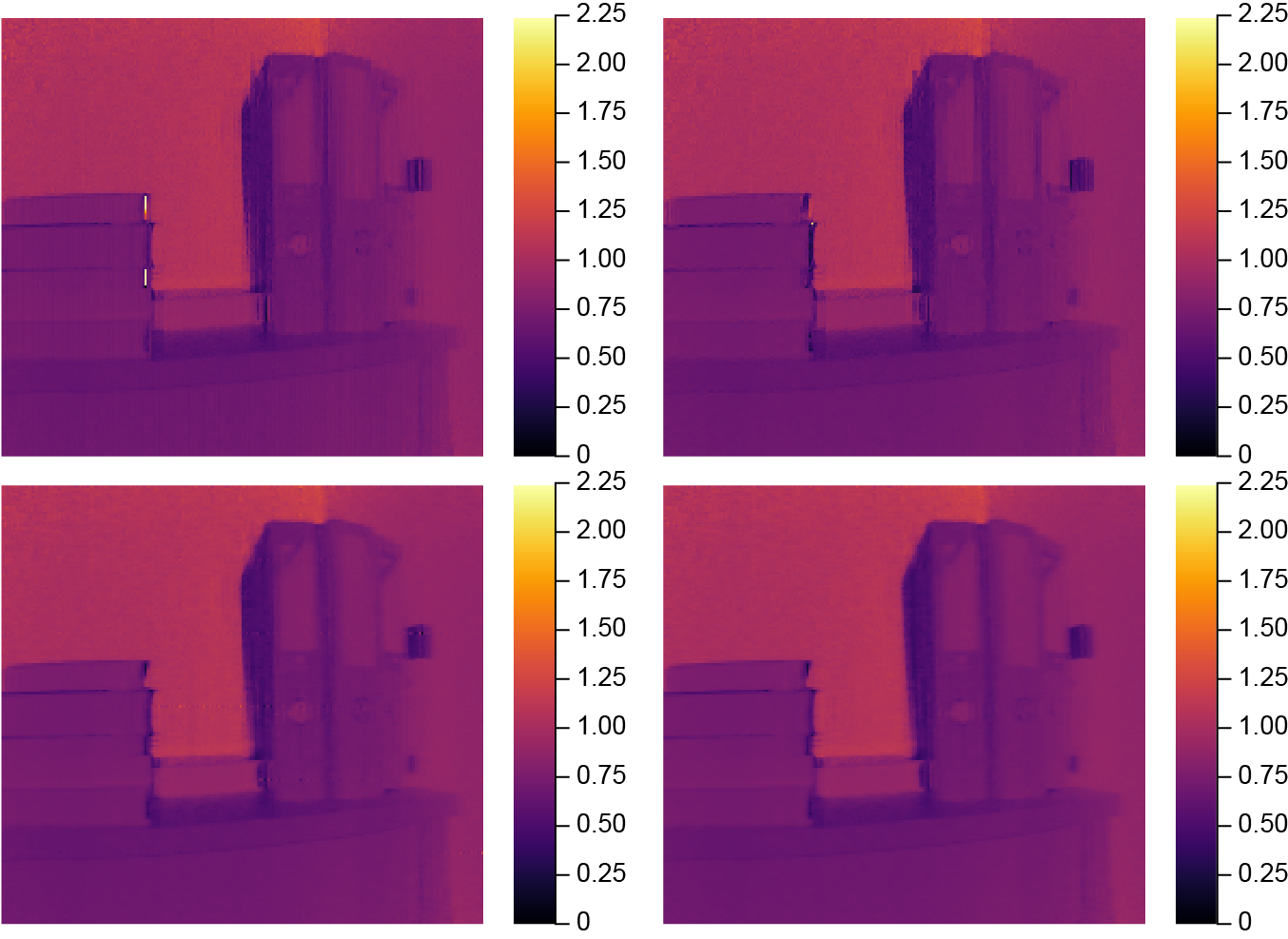}
    \caption{\textbf{Reconstructed depth images with a
        compression ratio of $2$.}
        Top left:  Block-wise $\ell_1$-minimization.
        Top  right: Global $\ell_1$-minimization.
        Bottom  left: Block-wise  TV-regularization.
        Bottom right: Global TV-regularization.
        The original depth image is shown in Figure~\ref{fig:depth1}.}
    \label{fig:depth-recon}
\end{figure}

\subsection{Numerical results}

For the first set of experiments, we consider phase difference images of a scenery with a couple of books
and folders shown in Figure~\ref{fig:phase1}, which is less than 1.2 meters away from
the camera (see Figure~\ref{fig:depth1}).
The  measured phase images  have been calibrated by subtracting
a reference image of constant distance (see  \cite{conde2017compressive} for additional calibration techniques).
Figure~\ref{fig:phase-recon} shows the reconstructed phase difference images from
compressed sensing data using a compression ratio of $2$
using block-wise and global $\ell_1$-minimization \eqref{eq:block_tikhonov}
and block-wise and global TV-regularization~\eqref{eq:blockTV}.
The corresponding depth reconstructions are shown in
Figure~\ref{fig:depth-recon}.

\begin{figure}[htb!]
    \centering
    \includegraphics[width=0.75\columnwidth]{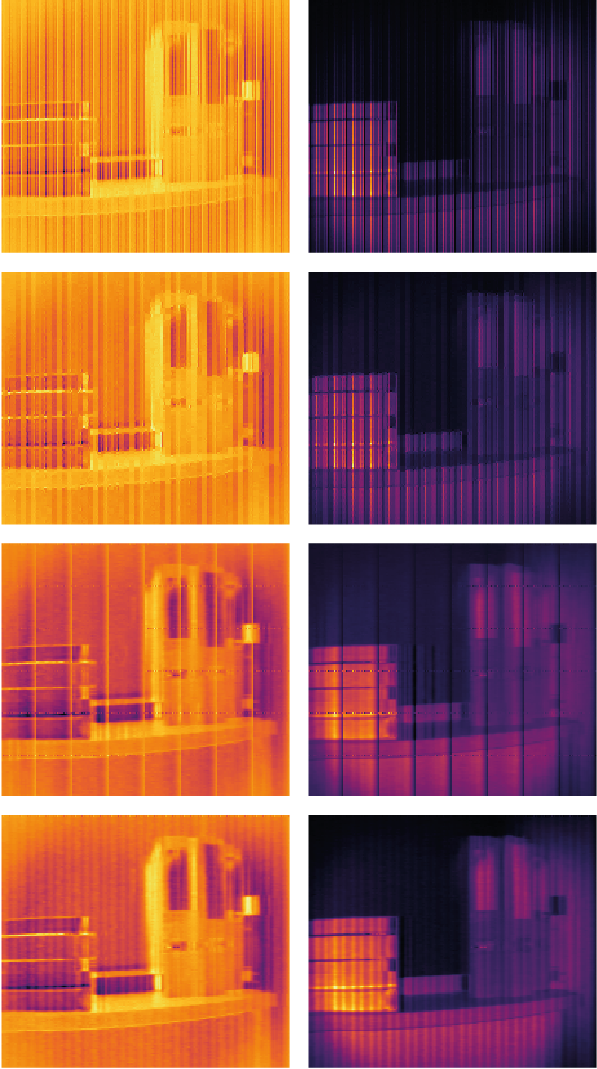}
    \caption{\textbf{Reconstructed phase differences with a
        compression ratio of $4.7$.}
        First row: Block-wise $\ell_1$-minimization.
        Second row: Global $\ell_1$-minimization.
        Third row:  Block-wise TV-regularization.
        Fourth row: Global TV-regularization.
        The original phases are shown in Figure~\ref{fig:phase1}}
    \label{fig:phases-lim}
\end{figure}

\begin{figure}[htb!]
    \centering
    \includegraphics[width=\columnwidth]{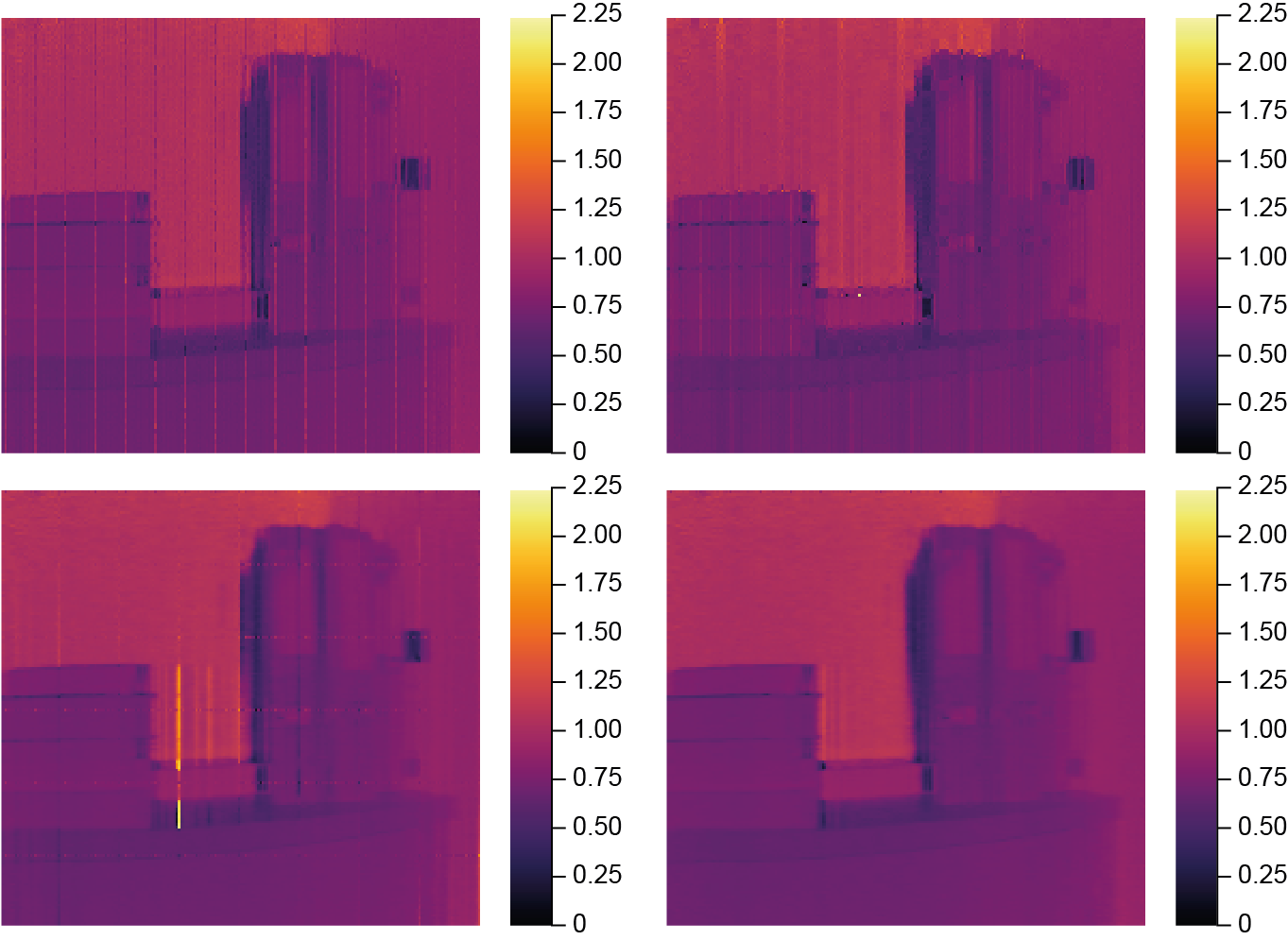}
    \caption{\textbf{Reconstructed depth images with a
        compression ratio of $4.7$.}
        Top left:  Block-wise $\ell_1$-minimization.
        Top  right: Global $\ell_1$-minimization.
        Bottom  left: Block-wise  TV-regularization.
        Bottom right: Global TV-regularization.
        The original depth image is shown in Figure~\ref{fig:depth1}.}
    \label{fig:depth-lim}
\end{figure}

To investigate the reconstruction quality when only using a very small amount
of data, we generated a measurement matrix with $m=3$. This results in a compression ratio 
$14/3$ of around $4.7$. In this example, we also increased the probability for zeros to $2/3$ and the resulting matrix had around $\SI{57}{\percent}$ zeros.
This means that the images can be captured very quickly since zero entries in the measurement matrix imply that the camera can skip the corresponding pixel.
The reconstructed  phase difference images using block-wise and global $\ell_1$-minimization and block-wise and global TV-regularization~\eqref{eq:blockTV} are shown in
Figure~\ref{fig:phases-lim} and the corresponding depth reconstructions
are shown in  Figure~\ref{fig:depth-lim}.

\subsection{Discussion}

Inspection of Figures~\ref{fig:phase-recon} and \ref{fig:depth-recon}
shows  that for a compression ratio  of $2$  all
reconstruction methods perform well. Opposed to that, for the
compression  ratio  $4.7$,  global TV-regularization clearly performs
best.  In a more quantitative way, this is demonstrated in Table~\ref{tab:errors},
where the  RMAE and PSNR are shown for depth image reconstructed
with the various reconstruction methods. From  that table, we see that
global TV-regularization in any case  yields the smallest  RMAE and largest PSNR.
For example, for the  compression ratio of 4.7,  global TV-regularization yields a
PSNR of 31.5, opposed to a PSNR of 26.4, 28.0 and 27.3
for block wise $\ell^1$-minimization,  global $\ell^1$-minimization
and block wise TV-regularization.

To investigate the dependence of the reconstruction error on the compression
ratio more closely,  we perform a series of compressed sensing  measurements and
reconstructions using  a set of 24  test images for
various compression ratios. The sceneries consist of the  books-image and other similar sceneries
captured with the ToF camera in an office and an apartment. In Figure~\ref{fig:error}, we show the
resulting average  RMAE and PSNR using   block-wise and global $\ell_1$-minimization  and
block-wise and global TV-regularization. For all test images we  use the same FISTA  and TV
 parameters  as above. In any case, global TV-regularization
 clearly outperforms  the other reconstruction methods.
This is especially prominent in the case  of a small number of measurements.
Summarizing these findings  we can clearly
recommend  global TV-regularization among the tested sparse recovery algorithms
for recovering the depth image from the proposed ToF camera design. 
Using tools such as dictionary learning or deep learning  to find optimal sparsifying 
transforms is an interesting line of future research.

 The reconstruction times for all method are comparable:
 On a laptop  with CPU Intel i5-3427U @ 1.80GHz,
 performing 100 iterations   took about
 $\SI{0.27}{s}$  for block-wise $\ell_1$,
 $\SI{0.42}{s}$  for global $\ell_1$,
$\SI{0.25}{s}$  for block-wise  TV, and
$\SI{0.33}{s}$ for global TV.  Finally,  note that we also
tested to separately recover the four phase images, which  we found to  yield
similar  results as for the difference images. Since the latter requires only half of the numerical
computations we suggest using sparse recovery of the  difference images
 when computational resources are limited.

\begin{table}
\centering
    \begin{tabular}{|c|c|c|c|}
       \toprule
        Method  & RMAE $(\SI{}{\percent})$ & PSNR ($\SI{}{\decibel}$) & CR \\
        \midrule
         FISTA block     & 1.3 & 29.8 & 2 \\
         FISTA global    &    1.2 & 32.9 & 2 \\
         TV block          &  1.0 & 34.4 & 2 \\
         TV global        & 0.9 & 34.9 & 2 \\
         \midrule
         FISTA block     &  2.7 & 26.4 & 4.7  \\
         FISTA global    &  2.3 & 28.0 & 4.7  \\
         TV block          &  1.8 & 27.3 & 4.7    \\
         TV global       &  1.4 & 31.5 & 4.7 \\
         \bottomrule
    \end{tabular}
    \caption{Reconstruction errors and signal to noise ratios
    for reconstructing the book depth image with compression ratio of
    $2$ and  $4.7$, respectively.}
    \label{tab:errors}
\end{table}

\begin{figure}[htb!]
    \centering
    \includegraphics[width=0.49\columnwidth]{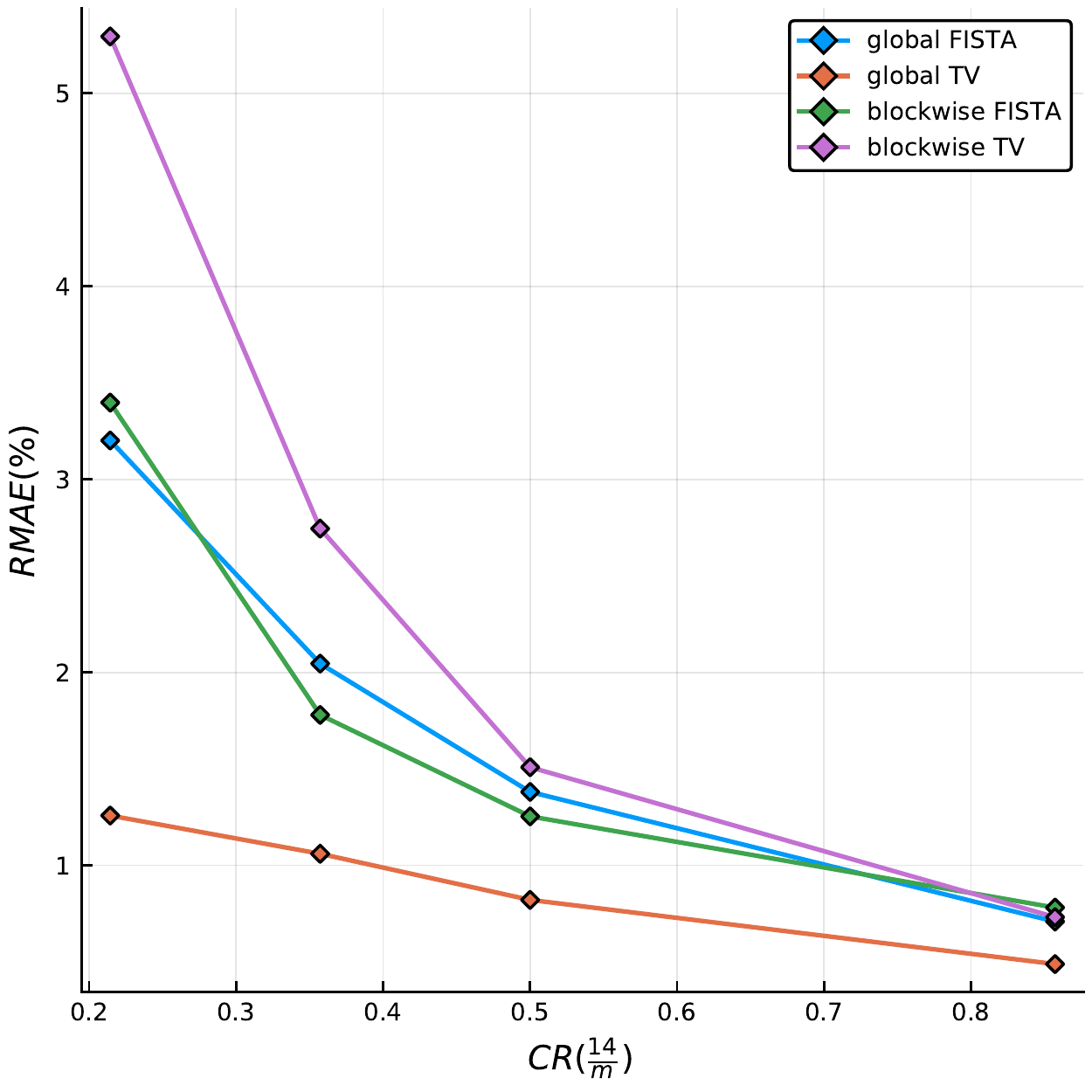}
    \includegraphics[width=0.49\columnwidth]{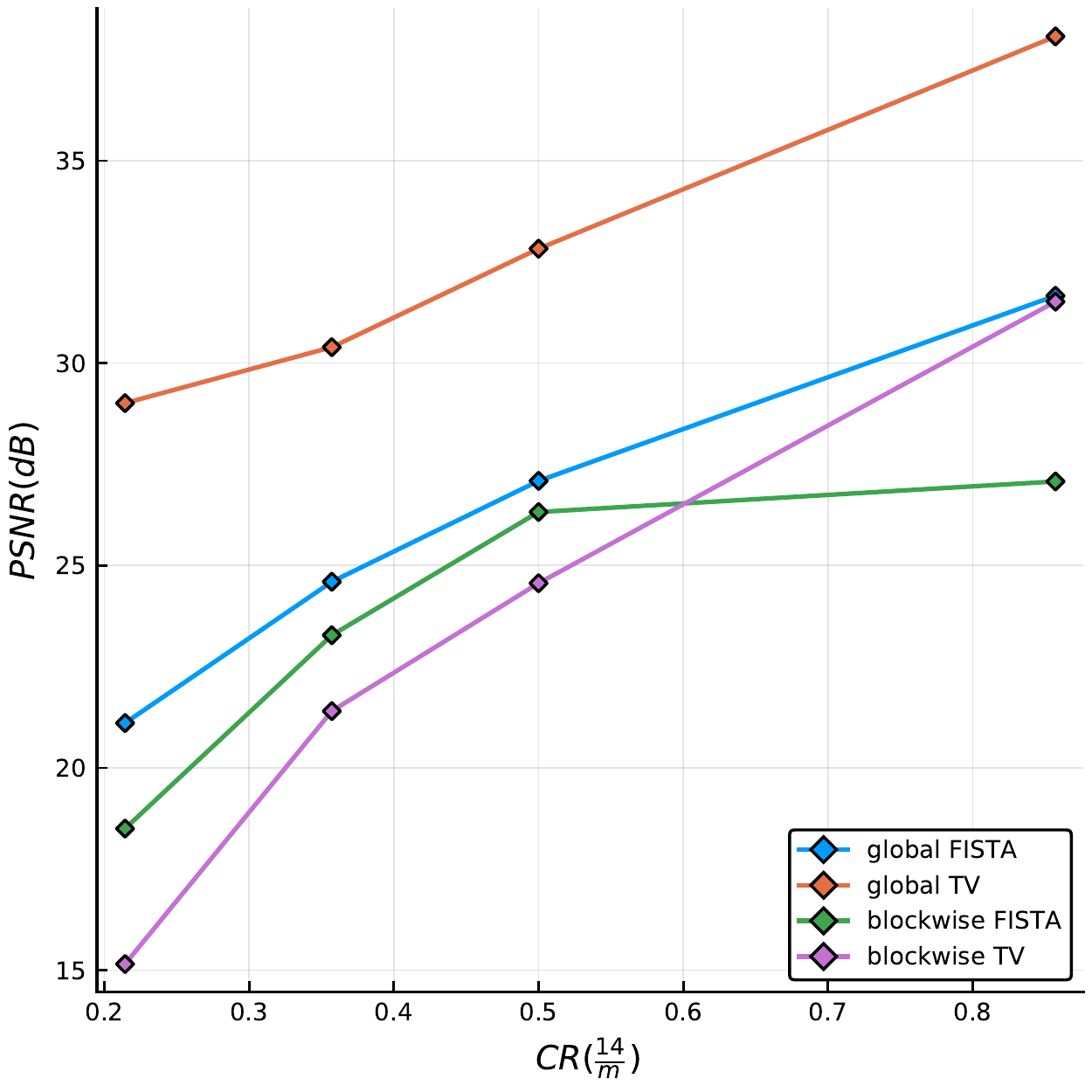}
\caption{  \textbf{Average error analysis.}
Left: Relative MAE error of the reconstructed depth images
    depending on the compression ratio $16/m$.
Right: Same as on left, but now using the PSNR for evaluating the reconstruction quality.} \label{fig:error}
\end{figure}

\section{Conclusion}

In this paper, we proposed a compressive ToF camera design that reduces the required amount of data to be read-out and
transferred. The proposed compressed ToF camera uses measurements within rows of the image which
yields a block-diagonal measurement matrix. Random partial circulant matrices as diagonal
blocks have
been shown to be compatible with current camera architecture. Their asymptotic recovery
guarantees do not directly apply in small block sizes. One can increase the block size, which is not really practical for
the ToF camera.  To overcome this issue, we proposed and implemented
strategy to increase the compressed sensing ability of the random partial circulant matrices.
Our experimental results clearly demonstrate that it is possible to
recover the original images from small measurement blocks.

For image reconstruction we used either blocks-wise or global reconstruction,   that both allow to exploit
the sparsity of the phase images in the two dimensional  wavelet basis ($\ell_1$-minimization), or
sparsity of the two dimensional gradient (TV-regularization).
We empirically  found TV-regularization to outperform  the wavelet-based $\ell_1$-reconstruction which a comparable numerical effort.
Future work will  be done to further improve the image quality and to
increase the reconstruction speed. Among others,  for that purpose, we will
investigate the use of machine learning in compressed sensing~\cite{mousavi2017learning}.

\section*{Acknowledgement}

S. Antholzer  and M. Haltmeier acknowledge support of the Austrian Science Fund (FWF),
project number 30747-N32. The work of M. Sandbichler has been supported by the
FWF, project number Y760.


\begin{thebibliography}{10}

\bibitem{Albrecht}
M.~Albrecht.
\newblock {\em Untersuchung von Photogate-PMD-Sensoren hinsichtlich
  qualifizierender Charakterisierungsparameter und -methoden}.
\newblock PhD thesis, University of Siegen, 2007.

\bibitem{antholzer2017nonlinear}
S.~Antholzer.
\newblock Nonlinear compressive time-of-flight {3D} imaging.
\newblock Master's thesis, University of Innsbruck, 2017.

\bibitem{antholzer2017compressive}
S.~Antholzer, C.~Wolf, M.~Sandbichler, M.~Dielacher, and M.~Haltmeier.
\newblock Compressive time-of-flight imaging.
\newblock In {\em Sampling Theory and Applications (SampTA), 2017 International
  Conference on}, pages 556--560. IEEE, 2017.

\bibitem{james}
D.~Aschenbr\"ucker.
\newblock Sparse recovery with random convolutions.
\newblock Master's thesis, University of Bonn, 2015.

\bibitem{fista}
A.~Beck and M.~Teboulle.
\newblock A fast iterative shrinkage-thresholding algorithm for linear inverse
  problems.
\newblock {\em SIAM J. Imaging Sci.}, 2(1):183--202, 2009.

\bibitem{explicit_RIP}
J.~Bourgain, S.~Dilworth, K.~Ford, S.~Konyagin, and D.~Kutzarova.
\newblock Explicit constructions of {RIP} matrices and related problems.
\newblock {\em Duke Math. J.}, 159(1):145--185, 2011.

\bibitem{cai2014sparse}
T.~Cai and A.~Zhang.
\newblock Sparse representation of a polytope and recovery of sparse signals
  and low-rank matrices.
\newblock {\em IEEE Trans. Inf. Theory}, 60:122 -- 132, 2014.

\bibitem{candes2006near}
E.~Candes and T.~Tao.
\newblock Near-optimal signal recovery from random projections: Universal
  encoding strategies?
\newblock {\em {IEEE} Trans. Inf. Theory}, 52(12):5406--5425, 2006.

\bibitem{CanEldNeeRan11}
E.~J. Cand{\`e}s, Y.~C. Eldar, D.~Needell, and P.~Randall.
\newblock Compressed sensing with coherent and redundant dictionaries.
\newblock {\em Appl. Comput. Harmon. Anal.}, 31(1):59--73, 2011.

\bibitem{CanRomTao06b}
E.~J. Cand{\`e}s, J.~K. Romberg, and T.~Tao.
\newblock Stable signal recovery from incomplete and inaccurate measurements.
\newblock {\em Comm. Pure Appl. Math.}, 59(8):1207--1223, 2006.

\bibitem{chambolle2011first}
A.~Chambolle and T.~Pock.
\newblock A first-order primal-dual algorithm for convex problems with
  applications to imaging.
\newblock {\em J. Math. Imaging Vision}, 40(1):120--145, 2011.

\bibitem{conde2017compressive}
M.~H. Conde.
\newblock {\em Compressive Sensing for the Photonic Mixer Device: Fundamentals,
  Methods and Results}.
\newblock Springer Vieweg, 2017.

\bibitem{Droeschel}
D.~Droeschel, D.~Holz, and S.~Behnke.
\newblock Multi-frequency phase unwrapping for time-of-flight cameras.
\newblock In {\em 2010 IEEE/RSJ International Conference on Intelligent Robots
  and Systems}, pages 1463--1469, 2010.

\bibitem{ToF-paper}
S.~Foix, G.~Alenya, and C.~Torras.
\newblock Lock-in time-of-flight (tof) cameras: A survey.
\newblock {\em IEEE Sensors Journal}, 11(9):1917--1926, 2011.

\bibitem{CS}
S.~Foucart and H.~Rauhut.
\newblock {\em A mathematical introduction to compressive sensing}.
\newblock Applied and Numerical Harmonic Analysis. Birkh\"auser/Springer, New
  York, 2013.

\bibitem{grasmair2011necessary}
M.~Grasmair, M.~Haltmeier, and O.~Scherzer.
\newblock Necessary and sufficient conditions for linear convergence of
  $\ell^1$-regularization.
\newblock {\em Comm. Pure Appl. Math.}, 64(2):161--182, 2011.

\bibitem{Hal13b}
M.~Haltmeier.
\newblock Stable signal reconstruction via $\ell^1$-minimization in redundant,
  non-tight frames.
\newblock {\em IEEE Trans. Signal Process.}, 61(2):420--426, 2013.

\bibitem{Hansard}
M.~E. Hansard, S.~Lee, O.~Choi, and R.~Horaud.
\newblock {\em Time-of-Flight Cameras -- Principles, Methods and Applications}.
\newblock Springer Briefs in Computer Science. Springer, 2013.

\bibitem{Howland2011Lidar3D}
G.~A. Howland, P.~Zerom, R.~W. Boyd, and J.~C. Howell.
\newblock Compressive sensing lidar for {3D} imaging.
\newblock In {\em CLEO: 2011 - Laser Science to Photonic Applications}, pages
  1--2, May 2011.

\bibitem{simple_det_mm}
M.~A. Iwen.
\newblock Simple deterministically constructible rip matrices with sublinear
  {F}ourier sampling requirements.
\newblock In {\em 2009 43rd Annual Conference on Information Sciences and
  Systems}, pages 870--875, 2009.

\bibitem{CMOS2009}
L.~Jacques, P.~Vandergheynst, A.~Bibet, V.~Majidzadeh, A.~Schmid, and
  Y.~Leblebici.
\newblock Cmos compressed imaging by random convolution.
\newblock In {\em Acoustics, Speech and Signal Processing, 2009 International
  Conference on}, pages 1113--1116. IEEE, 2009.

\bibitem{Kadambi2015Lidar3D}
A.~Kadambi and P.~T. Boufounos.
\newblock Coded aperture compressive 3-d lidar.
\newblock In {\em 2015 IEEE International Conference on Acoustics, Speech and
  Signal Processing (ICASSP)}, pages 1166--1170, April 2015.

\bibitem{katic2015compressive}
N.~Katic, M.~H. Kamal, A.~Schmid, P.~Vandergheynst, and Y.~Leblebici.
\newblock Compressive image acquisition in modern {CMOS} {IC} design.
\newblock {\em Int. J. Circ. Theor. and App.}, 43(6):722--741, 2015.

\bibitem{other-approach}
A.~Kirmani, A.~Cola{\c c}o, F.~N.~C. Wong, and V.~K. Goyal.
\newblock Codac: A compressive depth acquisition camera framework.
\newblock In {\em 2012 IEEE International Conference on Acoustics, Speech and
  Signal Processing (ICASSP)}, pages 5425--5428, 2012.

\bibitem{michaelTV17}
F.~Krahmer, C.~Kruschel, and M.~Sandbichler.
\newblock Total variation minimization in compressed sensing.
\newblock {\em arXiv preprint arXiv:1704.92195}, 2017.

\bibitem{suprema}
F.~Krahmer, S.~Mendelson, and H.~Rauhut.
\newblock Suprema of chaos processes and the restricted isometry property.
\newblock {\em Comm. Pure Appl. Math.}, 67(11):1877--1904, 2014.

\bibitem{redundant_dictionaries}
F.~Krahmer, D.~Needell, and R.~Ward.
\newblock Compressive sensing with redundant dictionaries and structured
  measurements.
\newblock {\em SIAM J. Math. Anal.}, 47(6):4606--4629, 2015.

\bibitem{krahmer2014stable}
F.~Krahmer and R.~Ward.
\newblock Stable and robust sampling strategies for compressive imaging.
\newblock {\em IEEE Trans. Image Process.}, 23(2):612--622, 2014.

\bibitem{Lange}
R.~Lange.
\newblock {\em {3D} Time-of-Flight Distance Measurement with Custom Solid-State
  Image Sensors in {CMOS}/{CCD}-Technology}.
\newblock PhD thesis, University of Siegen, 2000.

\bibitem{li2013convolutional}
K.~Li, L.~Gan, and C.~Ling.
\newblock Convolutional compressed sensing using deterministic sequences.
\newblock {\em {IEEE} Trans. Signal Process.}, 61(3):740--752, 2013.

\bibitem{mousavi2017learning}
A.~Mousavi and R.~G. Baraniuk.
\newblock Learning to invert: Signal recovery via deep convolutional networks.
\newblock {\em arXiv:1701.03891}, 2017.

\bibitem{poon2015role}
C.~Poon.
\newblock On the role of total variation in compressed sensing.
\newblock {\em SIAM J. Imaging Sci.}, 8(1):682--720, 2015.

\bibitem{rauhut2012restricted}
H.~Rauhut, J.~Romberg, and J.~A. Tropp.
\newblock Restricted isometries for partial random circulant matrices.
\newblock {\em Appl. Comput. Harmon. Anal.}, 32(2):242--254, 2012.

\bibitem{RauSchVan08}
H.~Rauhut, K.~Schnass, and P.~Vandergheynst.
\newblock Compressed sensing and redundant dictionaries.
\newblock {\em IEEE Trans. Inf. Theory}, 54(5):2210 --2219, 2008.

\bibitem{SchGraGroHalLen09}
O.~Scherzer, M.~Grasmair, H.~Grossauer, M.~Haltmeier, and F.~Lenzen.
\newblock {\em Variational methods in imaging}, volume 167 of {\em Applied
  Mathematical Sciences}.
\newblock Springer, New York, 2009.

\bibitem{Wolf}
C.~Wolf.
\newblock Framework for compressed sensing in time-of-flight based {3D}
  imaging.
\newblock Master's thesis, University of Innsbruck, 2016.

\bibitem{zimmermann1997low}
R.~Zimmermann and W.~Fichtner.
\newblock Low-power logic styles: Cmos versus pass-transistor logic.
\newblock {\em {IEEE} J. Solid-State Circuits}, 32(7):1079--1090, 1997.

\end{thebibliography}
\end{document}